\documentclass[preprint,11pt]{article}
\usepackage{amsmath}
\usepackage{amsthm}
\usepackage{amsfonts}
\usepackage[latin5]{inputenc}
\begin{document}
\numberwithin{equation}{section}
\newcounter{thmcounter}
\newcounter{Remarkcounter}
\newcounter{Examplecounter}
\numberwithin{thmcounter}{section}

\newtheorem{Prop}[thmcounter]{Proposition}
\newtheorem{Corol}[thmcounter]{Corollary}
\newtheorem{theorem}[thmcounter]{Theorem}
\newtheorem{Lemma}[thmcounter]{Lemma}
\theoremstyle{remark}
\newtheorem{Remark}[Remarkcounter]{Remark}
\newtheorem{Example}[Examplecounter]{Example}

\author{Nurettin Cenk Turgay\footnote{
e-mail: turgayn@itu.edu.tr, Adress: Istanbul Technical University, Faculty of Science and Letters,
Department of  Mathematics, 34469 Maslak, Istanbul, Turkey}}

\title{On the Lorentzian minimal surfaces in  $\mathbb E^4_1$ with finite type Gauss map}
\date{\today}
\maketitle
\begin{abstract}
In this paper, we study the Lorentzian minimal surfaces in the Minkowski space-time with finite type Gauss map. First, we obtain the classification of this type of surfaces with pointwise 1-type Gauss map. Then, we proved that there are no Lorentzian minimal surface in the Minkowski space-time with null 2-type Gauss map.

\end{abstract}
\section{Introductions}
The notion of finite type mappings defined on the submanifolds of semi-Euclidean spaces has been extensively studied by several geometers after it was introduced by B. Y. Chen in late 1970's. Let $\mathbb E^m_s$ denote the semi-Euclidean space with dimension $m$ and index $s$ whose metric tensor is given by 
$$\tilde g=\langle\ ,\ \rangle=-\sum\limits_{i=1}^sx_i^2+\sum\limits_{j=s+1}^mx_j^2$$
and consider a submanifold $M$ of  the semi-Euclidean space $\mathbb E^m_s$. A smooth mapping $\phi$ defined on $M$ into another semi-Euclidean space  $\mathbb E^N_S$ is  said to be $k$-type if it can be expressed as a sum of finitely many eigenvectors of the Laplace operator of $M$, \cite{ChenKitap,ChenMakale1986,ChenRapor}. Many important results about finite type mappings defined on semi-Riemannian submanifolds have appeared so far (cf. \cite{Bleecker-Weiner,ReillyHelvaci}).

In particular, the Gauss map of submanifolds has been worked in several articles in this direction after  some results on the submanifolds with 1-type Gauss map or 2-type Gauss map had been given in \cite{Chen-Piccinni}.
The Gauss map $\nu$ of $M$ is said to be $k$-type if it can be expressed as a sum of
\begin{equation}\label{FiniteTypeGaussDef}
 \nu=\nu_0+\nu_1+\nu_2+\hdots+\nu_k,
\end{equation}
where $\nu_0$ is a constant vector and $\nu_i$ is a non-constant eigenvector of $\Delta$ corresponding to eigenvalue $\lambda_i$ for $i=1,2,\hdots,k$ with  $\lambda_1<\lambda_2<\hdots< \lambda_k$ and $\Delta$  is the  Laplace operator of $M$ with respect to the induced metric of $M$. In addition, if one of these eigenvalues is zero, then  $\nu$ is said to be a null $k$-type mapping.

However, the Laplacian of  the Gauss map of several surfaces and hypersurfaces such as  
helicoids  of  the  1st, 2nd, and 3rd  kind, conjugate  Enneper's  surface of the second kind and 
 B-scrolls  in  a 3-dimensional Minkowski space $\mathbb E^3_1$,  
 generalized  catenoids, spherical n-cones, hyperbolical n-cones and 
 Enneper's hypersurfaces in $\mathbb E^{n+1}_1$   take the form
\begin{equation}\label{PW1TypeDefinition}
 \Delta \nu =f(\nu +C)
\end{equation}
for some smooth function $f$ on $M$ and some constant vector
$C$ (\cite{UDur2,Kim-Yoon}).
A submanifold   of a pseudo-Euclidean space is said to have
{\it pointwise 1-type Gauss map} if its  Gauss map satisfies
\eqref{PW1TypeDefinition} for some smooth  function $f$ on $M$  and some
constant vector $C$.   In particular, if $C$ is zero, it is 
said to be of {\it the first kind}.
Otherwise, it  is said to be of {\it the second kind} 
(cf. \cite{Arslan-and,CCK,KKKM,Choi-Kim,UDur,UDur3,Kim-Yoon-2}).

On the other hand, the theory of minimal and quasi-minimal surfaces is one of the most interesting topics in the semi-Euclidean geometry. A surface in $\mathbb E^m_s$ is said to be minimal or quasi-minimal if its mean curvature vector is zero or light-like respectively. In the very recent past, the classification of these type of surfaces are studied in some papers, in terms of type of their Gauss map, \cite{Dursun-Arsan,MinkowskiDursunTurgay, MILOUSHEVA01, NCTGenRelGrav}. In this paper, we focus on  Lorentzian minimal surfaces in the Minkowski space-time $\mathbb E^4_1$ with finite type Gauss map. In the Section 2, after we describe the notation that we will use in this paper, we give a short brief on the basic facts and definitions on the theory of submanifolds. In the Section 3, we focus on the Lorentzian surfaces with pointwise 1-type Gauss map. In the section 4, we study minimal Lorentzian surfaces with 2-type Gauss map.

The surfaces we are dealing with are smooth and connected unless otherwise stated.
\section{Prelimineries}

Let $M$ be an $n$-dimensional pseudo-Riemannian submanifold  of the 
pseu\-do-Euclidean space $\mathbb E^m_s$. 
We denote Levi-Civita connections of $\mathbb E^m_s$ and $M$ by $\widetilde{\nabla}$ and $\nabla$,  
respectively. The Gauss and Weingarten formulas are given, respectively, by
\begin{eqnarray}
\label{MEtomGauss} \widetilde\nabla_X Y&=& \nabla_X Y + h(X,Y),\\
\label{MEtomWeingarten} \widetilde\nabla_X \xi&=& -A_\xi(X)+D_X \xi,
\end{eqnarray}
 for any tangent vector field $X,\ Y$ and normal vector field $\xi$ on $M$, where $h$,  $D$  and  $A$ are the second fundamental form, the normal 
 connection and the shape operator of $M$, respectively. On the other hand, the shape operator $A$ and the second fundamental form $h$ of $M$ are related by
\begin{eqnarray}
\label{MinkAhhRelatedby} \langle A_\xi X,Y\rangle&=&\langle h(X,Y),\xi\rangle.
\end{eqnarray}
Relative null space at $p$ of $M$ is defined as
$$\mathcal N_p(M)=\{X\in T_pM|h(X,Y)=0, \ \mbox{for all} \ Y\in T_pM\}.$$ 
We say $M$ has degenerated relative null bundle if $(\mathcal N_p(M),\langle\ ,\ \rangle)$ is a degenerated inner product space for all $p\in M$.

The Gauss, Codazzi and Ricci equations are given, respectively, by
\begin{eqnarray}
\label{MinkGaussEquation} \langle R(X,Y,)Z,W\rangle&=&\langle h(Y,Z),h(X,W)\rangle-
\langle h(X,Z),h(Y,W)\rangle,\\
\label{MinkCodazzi} (\bar \nabla_X h )(Y,Z)&=&(\bar \nabla_Y h )(X,Z),\\
\label{MinkRicciEquation} \langle R^D(X,Y)\xi,\eta\rangle&=&\langle[A_\xi,A_\eta]X,Y\rangle,
\end{eqnarray}
where  $R,\; R^D$ are the curvature tensors associated with connections $\nabla$ 
and $D$, respectively, and 
$$(\bar \nabla_X h)(Y,Z)=D_X h(Y,Z)-h(\nabla_X Y,Z)-h(Y,\nabla_X Z).$$ 

Consider a Lorentzian surface $M$in $\mathbb E^4_1$ and let  $\{e_1,e_2;e_3,e_4\}$ be a positively oriented local orthogonal frame field on $M$ and $\{f_1,f_2\}$ the pseudo-orthogonal base field of the tangent bundle of $M$ given by $f_1=(e_1-e_2)/\sqrt 2$ and $f_2=(-e_1-e_2)/\sqrt 2$. Then, we have
\begin{eqnarray}
\label{LorSurfDelta}\Delta&=&f_1f_2+f_2f_1-\nabla_{f_1}f_2-\nabla_{f_2}f_1,\\
\label{LorSurfH}H&=& -h(f_1,f_2),\\
K&=&R(f_1,f_2,f_2,f_1)
\end{eqnarray}
where $\Delta$, $H$ and $K$ are the Laplace operator,  the mean curvature vector and the Gaussian curvature of $M$.

The smooth mapping
\begin{equation}\label{MinkGaussTasvTanim}
\begin{array}{rcl}\nu:M&\rightarrow&\subset \mathbb S^{5}_3(1)\subset \mathbb E^6_3\\
p&\mapsto&\nu(p)=(f_{1}\wedge f_{2})(p)\end{array}
\end{equation}
is called the (tangent) Gauss map of $M$. 

From \eqref{LorSurfDelta} one can obtain
\begin{eqnarray}
\label{Deltafu}\Delta(\phi \xi)&=&(\Delta \phi ) \xi+\phi  \Delta \xi-2\nabla\phi (\xi)
\end{eqnarray}
for any smooth function $\phi :M\rightarrow \mathbb R$ and any smooth mappings $\xi,\eta:M\rightarrow \mathbb E^6_3$, where $\nabla\phi$ is the gradient of $\phi$ defined by
$$\nabla\phi=-f_1(\phi) f_2-f_2(\phi) f_1.$$

We will use the following well-known lemmas, \cite{ONeillKitap}.
\begin{Lemma}\label{WellKnownLemmaLghtlks}
In a non-degenerated inner product space with index of 1, two light-like vector are lindearly dependent if and only if they are orthogonal.
\end{Lemma}

\begin{Lemma}\label{WellKnownLemmaNonDeg}
Let $U$ be a subspace of a non-degenerated inner product space $V$. U is non-degenerated if and only if $U\cap U^\perp=\{0\}.$
\end{Lemma}

The following lemma obtained in \cite{FuHou2010} is very useful.
\begin{Lemma}\cite{FuHou2010}\label{EMsLorSurfgf1f2}
Let $M$ be a Lorentzian surface in a semi Euclidean space $\mathbb E^m_s$. Then there exist local coordinates $(s,t)$ such that the induced metric is of the form of  
$$g=-m^2(dsdt+dsdt), \quad s\in I_1,\ t\in I_2,$$ where $m=m(s,t)$ is a non-vanishing function, $I_1$, $I_2$ are some open intervals. Moreover, the Levi-Civita connection of $M$ is given by
\begin{equation}
\label{EMsLorSurfgf1f2Levi} \nabla_{\partial_s}\partial_s= \frac{2m_s}{m}\partial_s,\quad\nabla_{\partial_s}\partial_t= 0,\quad\nabla_{\partial_t}\partial_t= \frac{2m_t}{m}\partial_t.
\end{equation}
\end{Lemma}

\section{Minimal Lorentzian surfaces and their Gauss map}

\begin{Lemma}\label{LemmaDeltaf1f2}
Let $M$ be a  Lorentzian minimal surface. Then its tangent and normal Gauss map $\nu$ and $\mu$ satisfy
\begin{equation}
\label{LemmaDeltaf1f2EQ} \Delta \nu= 2K \nu+2K^D \mu,
\end{equation}
where $K$ and $K^D$ are the Gaussian and the normal curvature, respectively.
\end{Lemma}
\begin{proof}
Consider a pseudo-orthogonal base field $\{f_1,f_2\}$  of the tangent bundle of $M$. As $M$ is minimal, we have $H=0$ from which and \eqref{LorSurfH} we have $h(f_1,f_2)=0$. By a direct calculation, we obtain
\begin{align}\label{LemmaDeltaf1f2Ara1}
\begin{split}
\Delta \nu=&f_1f_2(f_1\wedge f_2)+f_2f_1(f_1\wedge f_2)-\nabla_{f_1}f_2(f_1\wedge f_2)-\nabla_{f_2}f_1(f_1\wedge f_2)\\
=&f_1\left(f_1\wedge h \left(f_2,f_2\right)\right)+f_2\left(h \left(f_1,f_1\right)\wedge f_2\right)+\zeta_1f_1\wedge h \left(f_2,f_2\right)-\zeta_2h \left(f_1,f_1\right)\wedge f_2\\
=&\zeta_1f_1\wedge h(f_2,f_2)+h(f_1,f_1)\wedge h(f_2,f_2)-f_1\wedge A_{h(f_2,f_2)}f_1\\
&+f_1\wedge D_{f_1}h(f_2,f_2)- A_{h(f_1,f_1)}f_2\wedge f_2+D_{f_2}h(f_1,f_1)\wedge f_2\\
&-\zeta_2h(f_1,f_1)\wedge f_2+h(f_1,f_1)\wedge h(f_2,f_2)+\zeta_1f_1\wedge h \left(f_2,f_2\right)-\zeta_2h \left(f_1,f_1\right)\wedge f_2\\
=&2h(f_1,f_1)\wedge h(f_2,f_2)-\Big(f_1\wedge A_{h(f_2,f_2)}f_1+A_{h(f_1,f_1)}f_2\wedge f_2\Big)\\
&+\Big(D_{f_2}h(f_1,f_1)-2\zeta_2h(f_1,f_1)\Big)\wedge f_2+f_1\wedge \Big(D_{f_1}h(f_2,f_2)+2\zeta_1D_{f_1}h(f_2,f_2)\Big),
\end{split}
\end {align}
where $\zeta_i$ is a function defined by $\nabla_{f_i}f_1=\zeta_i f_1$ for $i=1,2$. As $h(f_1,f_2)=0$, Codazzi equation  \eqref{MinkCodazzi} implies 
\begin{equation}\label{LemmaDeltaf1f2Ara2}
D_{f_2}h(f_1,f_1)-2\zeta_2h(f_1,f_1)=D_{f_1}h(f_2,f_2)+2\zeta_1 D_{f_1}h(f_2,f_2)=0.
\end{equation}
On the other hand, using Gauss and Ricci equations \eqref{MinkGaussEquation} and \eqref{MinkRicciEquation}, we obtain
\begin{eqnarray}
\label{LemmaDeltaf1f2Ara3} \langle\Delta\nu,\nu\rangle&=&-2K\\
\label{LemmaDeltaf1f2Ara4} \langle\Delta\nu,\mu\rangle&=&2K^D.
\end{eqnarray}
From \eqref{LemmaDeltaf1f2Ara1}-\eqref{LemmaDeltaf1f2Ara4} we obtain \eqref{LemmaDeltaf1f2EQ}.
\end{proof}

\begin{Prop}\label{PropE41MinLor1stkind}
There exist two families of Lorentzian minimal surfaces in the Minkowski space $E^4_1$ with pointwise 1-type Gauss map of the first kind.
\begin{enumerate}
\item[(i)] A minimal surface lying in a Lorentzian hyperplane of $\mathbb E^4_1$;
\item[(ii)] A surface with degenerated relative null bundle.
\end{enumerate}
Conversely, every Lorentzian minimal surface with pointwise 1-type Gauss map of the first kind in the Minkowski space $E^4_1$ is congruent to an open portion of a surface obtained from these type of surfaces.
\end{Prop}

\begin{proof}
First, we prove these type of surfaces given above has pointwise 1-type Gauss map.
If $M$ is a minimal surface lying in a Lorentzian hyperplane $\Pi$ of $\mathbb E^4_1$, then \cite[Lemma 3.2.]{UDur2} implies that $M$ has pointwise 1-type Gauss map (also see \cite{UDur}). 
Now, let  $M$ be a surface with degenerated relative null bundle. Then, there exists a local pseudo-orthonormal base field $\{f_1,f_2\}$ of the tangent bundle of $M$ such that $h(f_1,f_1)=h(f_1,f_2)=0$. Thus, $M$ is minimal and $h(f_1,f_1)\wedge h(f_2,f_2)=0$ from which and the equation $K^D e_3\wedge e_4=h(f_1,f_1)\wedge h(f_2,f_2)$ we obtain that $K^D\equiv 0$. Hence, \eqref{LemmaDeltaf1f2EQ} implies that $M$ has pointwise 1-type Gauss map.

Now, we want to prove the remaining part of the proposition.

Let $M$ be a Lorentzian  surface in $\mathbb E^4_1$ and $s,t$ be the local coordinates given in Lemma \ref{EMsLorSurfgf1f2}. Consider the  pseudo-orthogonal basis $\{f_1,f_2\}$ given by
$$f_1=\frac 1m \partial_s\quad\mbox{and}\quad f_2=\frac 1m \partial_t.$$
If we suppose that $M$ is minimal, i.e., $H\equiv0$, then \eqref{LorSurfH} implies that $h(f_1,f_2)=0.$ On the other hand, the Gauss map $\nu=f_1\wedge f_2$ of $M$ satisfies \eqref{LemmaDeltaf1f2EQ}.

Now, we assume that $M$ has pointwise 1-type Gauss map of the first kind. Then \eqref{PW1TypeDefinition} is satisfied for $C=0$. From \eqref{PW1TypeDefinition} and \eqref{LemmaDeltaf1f2EQ} we obtain $2K^D e_3\wedge e_4=0$ from which we get $h(f_1,f_1)\wedge h(f_2,f_2)=0$. Thus, $h(\partial_s,\partial_s)$ and $h(\partial_t,\partial_t)$ are linearly dependent. 

Let $x:I_1\times I_2 \rightarrow \mathbb E^4_1$ be an isometric immersion of $M$ and consider the functions 
$$\begin{array}{rcl}
\psi_1:I_1\times I_2&\rightarrow&\mathbb R\\
         (s_0,t_0)&\mapsto    &\langle h(\partial_s,\partial_s),h(\partial_s,\partial_s)\rangle\big|_{x(s_0,t_0)}
\end{array}
$$
and 
$$\begin{array}{rcl}
\psi_2:I_1\times I_2&\rightarrow&\mathbb R\\
         (s_0,t_0)&\mapsto    &\langle h(\partial_t,\partial_t),h(\partial_t,\partial_t)\rangle\big|_{x(s_0,t_0)}
\end{array}.
$$

\textit{Case (1):} $\psi_1\equiv0$ or $\psi_2\equiv0$. In this case, $M$ has degenerated relative null bundle.

\textit{Case (2):} $\psi_1\neq0$ and $\psi_2\neq0$. In this case, the initial value problems
$$\phi_1'=\psi_1(\phi_2)^{-1/4},\quad \phi_1(0)=s_0 $$
and 
$$\phi_2'=\psi_1(\phi_2)^{-1/4},\quad \phi_2(0)=t_0 $$
admit unique solutions, say $\phi_1$ and $\phi_2$, respecetively, where $s_0\in I_1$ and $t_0\in I_2$. Let $S,T$ be local coordinates given by $S=\phi_1(s)$ and $T=\phi_2(t)$. Then, we have 
$g=-\hat m^2(S,T)(dSdT+dTdS),$
where $\hat m(S,T)=m(\phi_1(s),\phi_2(t))$.
Moreover,  the normal vector fields $h(\partial_S,\partial_S)$ and $h(\partial_T,\partial_T)$ are linearly dependent and unit. Thus, we have 
\begin{equation}\label{TransSurfhsshtt}
h(\partial_S,\partial_S)=\pm h(\partial_T,\partial_T).
\end{equation}

Now, let $\{e_3,e_4\}$ be an orthonormal base field of normal bundle of $M$ with $e_3=h(\partial_S,\partial_S)$. From Codazzi equation \eqref{MinkCodazzi} we obtain $D_{\partial_T}h(\partial_S,\partial_S)=D_{\partial_S}h(\partial_T,\partial_T)=0$. Therefore \eqref{TransSurfhsshtt} implies that $D e_3=0$, i.e., $e_3$ is parallel. As $M$ has codimension of 2, $e_4$ is also parallel. Moreover, by using \eqref{MinkAhhRelatedby}, we obtain $A_4=0$ because of \eqref{TransSurfhsshtt}. Thus, we have $\widetilde\nabla e_4=0$, i.e., $e_4$ is constant. Hence, $M$ is contained in a hyperplane $\Pi$. As $e_4$ is space-like, $\Pi$ is Lorentzian.
\end{proof}
Next, we obtain the following proposition.
\begin{Prop}
Let $M$ be a Lorentzian minimal surface in $\mathbb E^4_1$. If $M$ has pointwise 1-type Gauss map, then it is of the first kind.
\end{Prop}
\begin{proof}
If $M$ is a Lorentzian minimal surface, then \eqref{LorSurfH} implies $h(f_1,f_2)=0$ from which and \eqref{MinkAhhRelatedby} we have $\langle A_3f_1,f_2\rangle=\langle A_4f_1,f_2\rangle=0$ for any pseudo-orthonormal frame field $\{f_1,f_2,e_3,e_4\}$. In addition, the Gauss map $\nu=f_1\wedge f_2$ of $M$ satisfies \eqref{LemmaDeltaf1f2EQ}.

Now, we assume that the Gauss map $\nu$ of $M$ satifies \eqref{PW1TypeDefinition} for $C\neq0$. From \eqref{PW1TypeDefinition} and \eqref{LemmaDeltaf1f2EQ}, we have 
\begin{equation}\label{THMFULLCLASSC}
C=C_{12}f_1\wedge f_2+C_{34}e_3\wedge e_4.
\end{equation}
As $C$ is a constant vector, we have $f_i(C)=0,\ i=1,2$ from which and \eqref{THMFULLCLASSC} we obtain $f_i(C_{12})=f_i(C_{34})=0$ and
\begin{subequations}\label{THMFULLCLASSfiCALL}
\begin{eqnarray}
\label{THMFULLCLASSf1C} C_{12}h(f_1,f_1)\wedge f_2&=&-C_{34}(-A_3f_1\wedge e_4+A_4f_1\wedge e_3 ), \\
\label{THMFULLCLASSf2C} C_{12}f_1\wedge h(f_2,f_2)&=&-C_{34}(-A_3f_2\wedge e_4+A_4f_2\wedge e_3 ).
\end{eqnarray}
\end{subequations}
Thus, $C_{12},\ C_{34}$ are constants. On the other hand, from \eqref{THMFULLCLASSfiCALL} we have 
\begin{equation}\label{THMFULLCLASSAra001}
C_{12}^2\langle h(f_1,f_1)\wedge f_2,f_1\wedge h(f_2,f_2) \rangle=C_{34}^2\langle -A_3f_1\wedge e_4+A_4f_1\wedge e_3,-A_3f_2\wedge e_4+A_4f_2\wedge e_3 \rangle.
\end{equation}
By a direct calculation, we have
\begin{subequations}\label{THMFULLCLASSAra002ALL}
\begin{eqnarray}
\langle h(f_1,f_1)\wedge f_2,f_1\wedge h(f_2,f_2) \rangle&=&\langle h(f_1,f_1),h(f_2,f_2) \rangle,\\
\langle A_3f_1\wedge e_4,A_4f_2\wedge e_3 \rangle&=&0,\\
\langle A_4f_1\wedge e_3,A_3f_2\wedge e_4&=&0,\\
\langle A_3f_1\wedge e_4,A_3f_2\wedge e_4 \rangle&=&-\langle h(f_1,f_1),e_3 \rangle\langle h(f_2,f_2),e_3 \rangle.\\
\langle A_4f_1\wedge e_3,A_4f_2\wedge e_3\rangle&=&-\langle h(f_1,f_1),e_4 \rangle\langle h(f_2,f_2),e_4 \rangle.
\end{eqnarray}
\end{subequations}

By combaining \eqref{THMFULLCLASSAra001}-\eqref{THMFULLCLASSAra002ALL}, we obtain
$$\left( C_{12}^2+C_{34}^2\right)\langle h(f_1,f_1),h(f_2,f_2) \rangle=0.$$
As $C\neq0$, from the last equation we have $h(f_1,f_1)$ and $h(f_2,f_2)$ are orthogonal. 

Consider the open subset $\mathcal U=\{p\in M| h(f_1,f_1)\neq0 \ \mbox{or} \ h(f_2,f_2)\neq0\}$ of $M$ is not empty and let $\{e_3,e_4\}$ be a local orthonormal base field of the normal bundle of $M$ such that $h(f_1,f_1)=\alpha_3 e_3$ {and} $h(f_2,f_2)=\alpha_4 e_4$ over $\mathcal U$, where $\alpha_3$ and $\alpha_4$  are some functions. From \eqref{THMFULLCLASSfiCALL}, we have
\begin{subequations}\nonumber
\begin{eqnarray}
\nonumber C_{12}\alpha_3f_2\wedge e_3&=&-C_{34}(-A_3f_1\wedge e_4+A_4f_1\wedge e_3 ), \\
\nonumber C_{12}\alpha_4f_1\wedge e_4&=&-C_{34}(-A_3f_2\wedge e_4+A_4f_2\wedge e_3 )
\end{eqnarray}
\end{subequations}
on $\mathcal U$. From these equations, we have $A_3f_1=A_4f_2=0$ on $\mathcal U$  which implies $h\Big|_\mathcal U=0$, because of \eqref{MinkAhhRelatedby}. However, this is a contradiction if  $\mathcal U$ is not empty. 

Therefore, we have $h(f_1,f_1)=0$ or $h(f_2,f_2)=0$ which yields that $M$ has degenerated relative null bundle. Thus, Proposition \ref{PropE41MinLor1stkind} implies $M$ has pointwise 1-type Gauss map of the first kind which leads a contradiction. 
\end{proof}

\begin{Prop}\label{THMDegNull}
Let $M$ be a Lorentzian surface in $\mathbb E^m_s$. Then $M$ has degenerated relative null bundle, if and only if it is congruent to the surface given by 
\begin{equation}\label{DegRelNullSpaceSurfPosVect}
x(s,t)=s\eta_0+\beta(t)
\end{equation}
where $\eta_0$ is a constant light-like vector and $\beta$ is a null curve in $\mathbb E^m_s$ with $\langle\eta_0,\beta(t)\rangle\neq0$.
\end{Prop}
\begin{proof}
Let $M$ be a Lorentzian surface in $\mathbb E^m_s$, $x$ its position vector and $\{s,t\}$ some local coordinates given in Lemma \ref{EMsLorSurfgf1f2} satisfying \eqref{EMsLorSurfgf1f2Levi}. Consider the tangent vector fields $f_1=\frac 1m\partial_s$ and $f_2=\frac 1m\partial_t$.

Now, assume that $\mathcal N_p(M)$ is degenerated for all $p\in M$. Because of Lemma \ref{WellKnownLemmaNonDeg},  we may assume $\mathcal N_p(M)=\mbox{span} \{f_1\}$. which implies $h(f_1,f_1)=h(f_1,f_2)=0$. From these equations and \eqref{EMsLorSurfgf1f2Levi}
we have 
$\widetilde\nabla_{\partial_s}\partial_s=\nabla_{\partial_s}\partial_s$ {and} $\widetilde\nabla_{\partial_s}\partial_t=0$
from which we obtain
$x_{ss}=2\frac{m_s}{m}x_s$  {and}  $x_{st}=0$
By integrating these equations and re-defining $s$ properly, we obtain that $M$ is congruent to the surface given by \eqref{DegRelNullSpaceSurfPosVect}. 
\end{proof}

By combaining all of the results given in this section, we state
\begin{theorem}\label{L4MinimalFullCLASSTHEOPW1Type}
Let $M$ be a Lorentzian minimal surface in $\mathbb E^4_1$. Also suppose that no open part of $M$ is contained in a hyperplane of $\mathbb E^4_1$. Then, the following conditions are logically equivalent
\begin{enumerate}
\item[(i)]   $M$ has pointwise 1-type Gauss map;
\item[(ii)]   $M$ has pointwise 1-type Gauss map of the first kind;
\item[(iii)]  $M$ has harmonic Gauss map;
\item[(iv)] $M$  has degenerated relative null bundle;
\item[(v)]  $M$  has flat normal bundle;
\item[(vi)] $M$  is congruent to the surface given by \eqref{DegRelNullSpaceSurfPosVect}
for  a constant light-like vector  $\eta_0\in\mathbb E^4_1$ and  a null curve $\beta$ in $\mathbb E^4_1$ satisfying $\langle\eta_0,\beta(t)\rangle\neq0$.
\end{enumerate}
\end{theorem}
We also want to state the following corollary of this theorem.
\begin{Corol}
A Lorentzian minimal surface in $\mathbb E^4_1$ has proper pointwise 1-type Gauss map if and only if it lies in an Lorentzian hyperplane of $\mathbb E^4_1$ and it has non-constant Gaussian curvature.
\end{Corol}

\section{Minimal surfaces with 2-type Gauss map}
\begin{Lemma}\label{LemmaDeltaSQRf1f2} 
Let $M$ be a  Lorentzian minimal surface. Then its tangent and normal Gauss map $\nu$ and $\mu$ satisfy
\begin{equation}\label{LemmaDeltaSQRf1f2EQ}
\Delta^2\nu=2\Big(\Delta K+2K^2-2{K^D}^2\Big)\nu+2\Big(\Delta K^D+4KK^D\Big)\mu -4(\nabla K)(\nu)-4(\nabla K^D)(\mu)
\end{equation}
where $K$ and $K^D$ are the Gaussian and the normal curvature, respectively.
\end{Lemma}

Let $M$ be a Lorentzian surface in the Minkowski space $\mathbb E^4_1$ and $\nu$ and $\mu$ its tangent and normal Gauss map. We assume that $\nu$ satisfies 
\begin{equation}\label{Delta2nufDnu}
\Delta^2\nu=f\Delta\nu
\end{equation}
for a smooth fuction $f$. We  note that 
\begin{equation}\label{KKDinLambda}
(\nabla K)\nu),(\nabla K^D)(\mu)\in \Lambda^\perp,
\end{equation}
 where $\Lambda=\mathrm{span}\{\nu,\mu\}.$ From \eqref{LemmaDeltaf1f2EQ} and \eqref{LemmaDeltaSQRf1f2EQ}-\eqref{KKDinLambda} we obtain
\begin{subequations}\label{Delta2nufDnuEqAll}
\begin{eqnarray}
\label{Delta2nufDnuEq01} \Delta K+2K^2-2{K^D}^2&=&fK,\\
\label{Delta2nufDnuEq02} \Delta K^D+4KK^D&=&fK^D,\\
\label{Delta2nufDnuEq03} (\nabla K)(\nu)+(\nabla K^D)(\mu)&=&0.
\end{eqnarray}
\end{subequations}
             
Now, we consider a pseudo orthogonal frame field $\{f_1,f_2,e_3,e_4\}$ on $M$. Then we have 
\begin{align}\nonumber
\begin{split}\nonumber
(\nabla K)(\nu)+(\nabla K^D)(\mu)=&\Big(f_1(K)h^3_{22}-f_1(K^D)h^4_{22}\Big)f_1\wedge e_3+ \Big(f_1(K)h^4_{22}+f_1(K^D)h^3_{22}\Big)f_1\wedge e_4\\
&-\Big(f_2(K)h^3_{11}+f_2(K^D)h^4_{11}\Big)f_2\wedge e_3- \Big(f_2(K)h^4_{11}-f_2(K^D)h^3_{11}\Big)f_2\wedge e_4
\end{split}
\end{align}
from which and \eqref{Delta2nufDnuEq03} we have 
\begin{subequations}\label{Delta2nufDnuEq02ALL}
\begin{eqnarray}
\label{Delta2nufDnuEq02a}\left(\begin{array} {cc}
h^3_{22}& -h^4_{22}\\
h^4_{22}&h^3_{22}
\end{array}\right)
\left(\begin{array} {cc}
f_1(K)\\
f_1(K^D)
\end{array}\right)
&=&0\\
\label{Delta2nufDnuEq02b}
\left(\begin{array} {cc}
h^3_{11}&h^4_{11}\\
h^4_{11}&-h^3_{11}
\end{array}\right)
\left(\begin{array} {c}
f_2(K)\\
f_2(K^D)
\end{array}\right)
&=&0
\end{eqnarray}
\end{subequations}

Now, consider the open subset $\mathcal U=\{p| \nabla K\neq0 \mbox{ or } \nabla K^D\neq0\}$ of $M$ and assume that it is non-empty, i.e. either $K$ or $K^D$ is non-constant. From \eqref{Delta2nufDnuEq02ALL}, we have either $({h^3_{11}})^2+({h^4_{11}})^2=0$ or $({h^3_{22}})^2+({h^4_{22}})^2=0$ on $\mathcal U$. Thus, we have $\langle h(f_1,f_1), h(f_2,f_2)\rangle=0$ and  $\langle [A_3,A_4]f_!,f_2 \rangle=0$ on $\mathcal U$. From these equations, \eqref{MinkGaussEquation} and \eqref{MinkRicciEquation} we obtain $K=K^D=0$ on $\mathcal U$ which yields a contradiction.

Therefore, we have $K$ and $K^D$ are constants. Note that if $K^D=0$, then \eqref{LemmaDeltaf1f2EQ} implies that $M$ has 1-type Gauss map. Therefore, from \eqref{Delta2nufDnuEq02} we have $f=4K$. Next, we put this equation into \eqref{Delta2nufDnuEq01} and obtain $K=K^D=0$. Hence, \eqref{LemmaDeltaf1f2EQ} implies $\Delta\nu=0.$ Hence, we have
\begin{theorem}
Let $M$ be a Lorentzian minimal surface in the Minkowski space $\mathbb E^4_1$. If the Gauss map $\nu$ of $M$ satisfies \eqref{Delta2nufDnu}, then $f$ is constant and $M$ is  1-type.
\end{theorem}

\begin{Corol}
There exists no Lorentzian minimal surface with biharmonic Gauss map in the Minkowski space $\mathbb E^4_1$.
\end{Corol}

\begin{Corol}
There exists no Lorentzian minimal surface with null 2-type Gauss map in the Minkowski space $\mathbb E^4_1$.
\end{Corol}

\bibliographystyle{plain} 
\bibliography{NCTurgay_Bibtex}

\end{document}